\documentclass[11pt]{article}

\usepackage{amsmath,amssymb,amsthm}
\usepackage{color}

\usepackage[colorlinks]{hyperref}
\usepackage{mathrsfs}
\usepackage{epsfig}
\usepackage{graphicx}

\newtheorem{theorem}{Theorem}[section]
\newtheorem{lemma}{Lemma}[section]
\newtheorem{proposition}{Proposition}[section]
\theoremstyle{remark}
\newtheorem{remark}{\hskip\parindent\bf Remark}[section]

\theoremstyle{definition}

\numberwithin{equation}{section}

\newcommand{\ba}{\begin{array} }
\newcommand{\ea}{\end{array} }
\newcommand{\be}{\begin{equation} }
\newcommand{\ee}{\end{equation} }

\begin{document}

\title{The Hitting Times of A Stochastic Epidemic Model}
\author{Qingshan Yang\\School of Mathematics and Statistics
                     \\Northeast Normal University, Changchun 130024,
                     China\\
Xuerong Mao\thanks{Correspoinding author. E-mail address:
x.mao@strath.ac.uk. This paper is supported by National Natural Science Funds of China (Grant No. 11401090) }
\\Department of Mathematics and Statistics
\\University of Strathclyde, Glasgow G11XH,
UK
}


\date{}

\maketitle

\begin{abstract}

In this paper, we focus on the hitting times of a stochastic
epidemic model presented by \cite{Gray}. Under the help of the
auxiliary stopping times, we investigate the asymptotic limits of
the hitting times by the variations of calculus and the large
deviation inequalities when the noise is sufficiently small. It can
be shown that the relative position between the initial state and
the hitting state determines the scope of the hitting times greatly.

\end{abstract}


\newcommand{\dx}{\Delta x}
\newcommand{\dt}{\Delta t}

\section{Introduction}
\label{sec:intro}

In \cite{Gray}, Gray, Greenhalgh, Hu, Mao and Pan discuss the
asymptotic dynamics of a stochastic SIS epidemic model. Especially,
they show the ergodic property and the recurrence of the model.
Recently, there are also some other papers concerned on the
ergodicity of stochastic epidemic models such as \cite{Liu},
\cite{Yang} e.t.c. In these papers, to obtain the ergodicity and the
recurrence, the noise is usually assumed to be small. According to
the theory of Markov processes, the recurrence implies that it can
reach any state in a finite time. Then another question arises: how
long will it take? In this paper, we will investigate the asymptotic
limits of the hitting times for any state for sufficiently small
noise. This study may be helpful to the investigation of the rate
under control of the disease transmission.

Firstly, let us recall some notations and results in \cite{Gray}.
Let $(\Omega, \mathcal{F}, \{\mathcal{F}_t\}_{t\geq0}, \mathbb{P})$
be a complete probability space with a filtration
$\{\mathcal{F}_t\}_{t\geq0}$ satisfying the usual conditions, and
$\{B(t), t\geq 0\}$ be a scalar standard Brownian motion defined on
the probability space. The stochastic version of the well known SIS
model is given by the following It\^{o} SDE
\begin{equation*}\label{SDE}
\left\{
\begin{aligned}
&dS(t)=[\mu N-\beta S(t)I(t)+\gamma I(t)- \mu S(t)]dt-\sigma S(t)I(t)dB(t),\\
&dI(t)=[\beta S(t)I(t)-(\mu+\gamma)I]dt+\sigma S(t)I(t)dB(t).
\end{aligned}
\right.
\end{equation*}
\textcolor[rgb]{0.00,0.00,1.00}{Given} that $S(t)+I(t)=N$,
\textcolor[rgb]{0.00,0.00,1.00}{it is sufficient to study} the
following SDE for $I(t)$
\begin{equation}\label{SDE I}
dI(t)=I(t)\left([\beta N-\mu-\gamma-\beta I(t)]dt+\sigma
(N-I(t))dB(t)\right)
\end{equation}
with initial value $I(0)=x\in(0, N)$. In \cite{Gray}, they showed
that if $R^S_0:=\frac{\beta
N}{\mu+\gamma}-\frac{\sigma^2N^2}{2(u+r)}>1$, then the SDE
\eqref{SDE I} obeys
$$
\limsup_{t\to\infty}I(t)\geq\xi, \quad
\liminf_{t\to\infty}I(t)\leq\xi, ~~a.s.,
$$
where
$\xi=\sigma^{-2}\left(\sqrt{\beta^2-2\sigma^2(\mu+\gamma)}-(\beta-\sigma^2
N)\right)$ and
$\lim\limits_{\sigma\to0}\xi=N-\frac{\mu+\gamma}{\beta}$ (Theorem
5.1 in \cite{Gray}). This showed some recurrence of the model: I(t)
will rise to or above the level $\xi$ infinitely often with
probability one.

In fact, they actually showed the ergodic property and recurrence
when $R^S_0>1$ (Theorem 6.2 in \cite{Gray}). That is to say, the SDE
\eqref{SDE I} can reach any point in $(0, N)$. According to some
other papers concerned on the ergodicity of stochastic epidemic
models such as \cite{Liu}, \cite{Yang}, the noise is usually assumed
to be small enough to obtain the ergodicity. Therefore, in this
paper, we are interested in the scopes of the hitting times when the
noise is sufficiently small, i.e., how long will it take to arrive
at any fixed point in $(0, N)$?

\textcolor[rgb]{0.00,0.00,1.00}{A question may arise: what about the
other cases of $\sigma$ when $R^S_0>1$? Actually, the problem
becomes much more complicated to solve and this paper is an attempt
to investigate the limits of the hitting times for sufficiently
small noise.}

To emphasize the dependence of $\sigma$, we will denote the solution
to \eqref{SDE I} by $I^\sigma(\cdot)$ throughout this paper.
Obviously, $I^0(\cdot)$ is the solution to the deterministic system.
Now, we will formulate our question in the recurrent condition and
assume that $R^D_0:=\frac{\beta N}{\mu+\gamma}>1$ for the sake of
the recurrence throughout this paper (Obviously, which is equivalent
to $R^S_0>1$, if $\sigma$ is sufficiently small). For any
$y\in(0,N)$, define
$$
\tau^\sigma_y:=\inf\{t\geq0; I^\sigma(t)=y\}.
$$
Clearly, $\tau^\sigma_y$ is a stopping time, and we will investigate
its asymptotic limit as $\sigma\to0$. Obviously, by Theorem 3.1 in
\cite{Gray}, $\tau^\sigma_0=\infty$, a.s., thus
$\tau^\sigma=\tau^\sigma_0\wedge\tau^\sigma_y$, which is the exit
time from $[0,y]$. Therefore, it is encouraged to consider the
problem of exit from $[0,y]$. But the model \eqref{SDE I} has a
degenerate diffusion coefficient at $0$, and starting from any
neighborhood of the characteristic boundary $0$, the hitting times
of the other points in the neighborhood of $0$ seem sufficiently
large, which does not satisfy the conditions for the exit problem
from a domain (\cite{Dembo}). Hence, we need to introduce the
auxiliary stopping times, and investigate their asymptotic limits
using the variations of calculus and the large deviation.

In this paper, we organize the sections as followed. In Section 2,
we will introduce our main results. In Section 3, we will give some
preliminaries used later. Section 4 will end this paper with the
proof of main results.

\section{Main results}
\label{sec:lists}

Firstly, we will give some symbols. Define
$*=N-\frac{\mu+\gamma}{\beta}$, and
\begin{equation*}
\begin{aligned}
\overline{V}_y&\triangleq\inf_{t>0}\inf_{u\in
L^2([0,t])}\left\{\frac{\int^t_0|u(s)|^2ds}{2}; \phi(t)=y, ~where~\right.\\
&\left.\phi(s)=*+\int^s_0\phi(\theta)\left[(\beta+\sigma
u(\theta))(N-\phi(\theta))-\mu-\gamma\right]d\theta\right\}.
\end{aligned}
\end{equation*}

\begin{theorem}\label{main results}
For any $x, y\in(0,N)$, if $I^\sigma(0)=I^0(0)=x$,
$T_y=\inf\{t\geq0; I^0(t)=y\}=\frac{1}{\beta
*}\ln\frac{y(x-*)}{x(y-*)}$, then for any $\delta>0$,

(1)~if~ $0<x<*$, $*<y<N$ or $0<y<x$, then
\begin{equation*}
\lim_{\sigma\to0}\mathbb{P}_x\left\{e^{\frac{\overline{V}_y-\delta}{\sigma^2}}<\tau^\sigma_y<e^{\frac{\overline{V}_y+\delta}{\sigma^2}}
\right\}=1,
\end{equation*}
and $0<\overline{V}_y<\infty$;

(2)~if~ $0<x<*$, $x\leq y<*$, then
\begin{equation*}
\lim_{\sigma\to0}\mathbb{P}_x\left\{|\tau^\sigma_y-T_y|>\delta\right\}=1,
\end{equation*}
and $T_y<\infty$;

(3)~if~$0<x<*$ or $*<x<N$, $y=*$, then
$\lim\limits_{\sigma\to0}\tau^\sigma_*=\infty$, and
$\lim\limits_{\sigma\to0}\sigma^2\ln\tau^\sigma_*=0$ in probability;

(4)~if~$*<x<N$, $x<y<N$ or $0<y<*$, then
\begin{equation*}
\lim_{\sigma\to0}\mathbb{P}_x\left\{e^{\frac{\overline{V}_y-\delta}{\sigma^2}}<\tau^\sigma_y<e^{\frac{\overline{V}_y+\delta}{\sigma^2}}
\right\}=1,
\end{equation*}
and $0<\overline{V}_y<\infty$;

(5)~if~$*<x<N$, $*<y<x$, then
\begin{equation*}
\lim_{\sigma\to0}\mathbb{P}_x\left\{|\tau^\sigma_y-T_y|>\delta\right\}=1,
\end{equation*}
and $T_y<\infty$.
\end{theorem}

\begin{remark}
By the results of \cite{Gray}, we know that $\tau^\sigma_y<\infty$
a.s. for any $y\in(0,N)$. But the scopes of the hitting times depend
on the relative position between the initial and the hitting states.
Take $0<x<*$ for an example. If $x<y<*$, then the hitting time
$\tau^\sigma_y$ approaches a fixed constant with a large probability
when the noise is small enough. But if $*<y<N$, the time to arrive
at $y$ is exponentially large about the noise $\sigma$ with a large
probability. This delicate description may help us understand the
disease transmission better.
\end{remark}

\section{Preliminaries}
\label{sec:eqn} \setcounter{equation}{0}

Before the proofs of main results, we will give some well known
results concerned on the problem of exit from a domain. The revelent
literature may be found in \cite{Day89}, \cite{Day92}, \cite{Duk87},
\cite{Fre85b} etc and references therein. In this paper, we suggest
\cite{Dembo} for reference.

Consider the SDE
\begin{equation}\label{stochastic system}
\left\{
\begin{aligned}
&dx^\varepsilon(t)=b(x^\varepsilon(t))dt+\sqrt{\varepsilon}\sigma(x^\varepsilon(t))d\omega(t),\\
&x^\varepsilon(t)\in\mathbb{R}^d, ~x^\varepsilon(0)=x,
\end{aligned}
\right.
\end{equation}
in the open, bounded domain $G\subset\mathbb{R}^d$, where $b(\cdot)$
and $\sigma(\cdot)$ are uniformly Lipschitz continuous functions of
appropriate dimensions and $\omega(\cdot)$ is a standard Brownian
motion.

Define the cost function
\begin{equation*}
\begin{aligned}
V(y,z,t)&\triangleq\inf\left\{I_{y,t}(\phi); \phi\in C([0,t]):
\phi(t)=z\right\}\\
&=\inf\left\{\frac{\int^t_0|u(s)|^2ds}{2}; u\in L^2([0,t]),
\phi(t)=z,\right.\\
&\left.
~where~\phi(s)=y+\int^s_0b(\phi(\theta))d\theta+\int^s_0\sigma(\phi(\theta))u(\theta)
d\theta\right\},
\end{aligned}
\end{equation*}
where $I_{y,t}(\cdot)$ is the good rate function of (5.5.26) in
\cite{Dembo}, which controls the LDP (large deviation principles)
associated with \eqref{stochastic system}.

Define
\begin{equation*}
V(y,z)\triangleq\inf_{t>0}V(y,z,t).
\end{equation*}

{\bf Assumption (A-1)} The unique stable equilibrium point in $G$ of
the d-dimensional ordinary differential equation
\begin{equation}\label{deterministic system}
\dot{\phi}(t)=b(\phi(t))
\end{equation}
is at $0\in G$, and
\begin{equation*}
\phi(0)\in G\Rightarrow\forall t>0, \phi(t)\in G ~and~
\lim_{t\to\infty}\phi(t)=0
\end{equation*}

{\bf Assumption (A-2)} All the trajectories of the deterministic
system \eqref{deterministic system} starting at $\phi(0)\in\partial
G$ converge to $0$ as $t\to\infty$.

{\bf Assumption (A-3)}
$\overline{V}\triangleq\inf\limits_{z\in\partial G}V(0,z)<\infty$.

{\bf Assumption (A-4)} There exists an $M<\infty$ such that, for all
$\rho>0$ small enough and all $x, y$ with $|x-z|+|y-z|\leq\rho$ for
some $z\in\partial G\cup\{0\}$, there is a function $u$ satisfying
that $||u||<M$ and $\phi(T(\rho))=y$, where
\begin{equation*}
\phi(t)=x+\int^t_0 b(\phi(s))ds+\int^t_0\sigma(\phi(s))u(s)ds
\end{equation*}
and $T(\rho)\to0$ as $\rho\to0$.

\begin{theorem}(Theorem 5.7.11 in \cite{Dembo})\label{Thm 5.7.11}
Assume (A1)-(A4). For all $x\in G$ and all $\delta>0$,
\textcolor[rgb]{0.00,0.00,1.00}{$\tau^\varepsilon=\inf\{t\geq0;
x^\varepsilon(t)\in\partial G\}$,}
\begin{equation*}
\lim_{\varepsilon\to0}\mathbb{P}\left\{e^{\frac{\overline{V}-\delta}{\varepsilon}}<\tau^\varepsilon<e^{\frac{\overline{V}+\delta}{\varepsilon}}\right\}=1.
\end{equation*}
\end{theorem}

Now, we turn to our proofs. We adapt the old symbols given above.
\textcolor[rgb]{0.00,0.00,1.00}{In our case, $(0,N)$ play the same
role as $R^d$ in (\ref{stochastic system}) and for any $y\in(0,N)$,
$*=N-\frac{\mu+\gamma}{\beta}$ is the positive equilibrium of the
deterministic model $I^0(\cdot)$ as $0$ in (\ref{stochastic
system}).} Note that $\overline{V}_y=V(*,y)$, and defne
$\overline{V}_\rho=V(*,\rho)$, $\overline{V}_{-\rho}=V(*,N-\rho)$
associated with the SDE \eqref{SDE I} and
$\tau^\sigma_\rho=\inf\{t\geq0; I^\sigma(t)=\rho\}$, and
$\tau^\sigma_{-\rho}=\inf\{t\geq0; I^\sigma(t)=N-\rho\}$.

\begin{lemma}\label{LDP lower bound} For any positive sequence $\{T_n,n\geq1\}$ such that $\sup\limits_n T_n<\infty$ and sufficiently small $\rho_n>0$, there
exists a $M>0$ such that
\begin{align*}\label{results from the lower bound}
\limsup_{\sigma\to0}\sigma^2\log\mathbb{P}_{*-\delta_0}\left\{\inf_{t\in[0,T_n]}I^\sigma(t)<2\rho_n,
I^\sigma(t)\in(\frac{\rho_n}{2},*-\frac{\delta_0}{2})\right\}
\leq-\frac{(\ln\rho^{-1}_n)^2}{8\sigma^2T_n M}.
\end{align*}
\end{lemma}
\begin{proof}
Note that
\begin{equation*}
\begin{aligned}
&d\ln\left((I^\sigma(t))^{-1}\right)=-d\ln I^\sigma(t)=\left\{\beta
I^\sigma(t)-\beta
N+\mu+\gamma+\frac{\sigma^2(N-I^\sigma(t))^2}{2}\right\}dt-\sigma(N-I^\sigma(t))dB(t),\\
&\ln\left(I^{-1}(0)\right)=\ln(*-\delta_0).
\end{aligned}
\end{equation*}
Therefore, there is $M>0$ such that
$\left|\ln\left((I^\sigma)^{-1}(t)\right)\right|\leq
MT_n+\sigma\sup\limits_{t\in[0,T_n]}|\int^t_0(N-I^\sigma(s))dB(s)|$
for $t\in[0,T_n]$. Since $\rho_n>0$ is sufficiently small, we may
assume without loss of generality that $M
T_n\leq\frac{\ln\rho^{-1}_n}{2}$, then
\begin{equation}\label{aux-1}
\begin{aligned}
&\mathbb{P}_{*-\delta_0}\left\{\inf_{t\in[0,T_n]}I^\sigma(t)<2\rho_n,
I^\sigma(t)\in(\frac{\rho_n}{2},*-\frac{\delta_0}{2})\right\}\\
&\leq\mathbb{P}_{*-\delta_0}\left\{
\sigma\sup\limits_{t\in[0,T_n]}\left|\int^t_0(N-I^\sigma(s))dB(s)\right|>\frac{\ln\rho^{-1}_n}{2}\right\}.
\end{aligned}
\end{equation}
Let
$M_n(t)=\exp\left\{\lambda\sigma\int^t_0\left(N-I^\sigma(s)\right)dB(s)-\frac{\lambda^2\sigma^2}{2}\int^t_0\left(N-I^\sigma(s)\right)^2ds\right\}$,
where $\lambda=\frac{(\ln\rho^{-1}_n)}{2\sigma^2T_n N^2}>0$. Then
$\{M_n(t), t\geq0\}$ is a sequence of martingale and
\begin{equation}\label{exponential martingale}
\begin{aligned}
&\mathbb{P}_{*-\delta_0}\left\{
\sigma\sup\limits_{t\in[0,T_n]}\left|\int^t_0\left(N-I^\sigma(s)\right)dB(s)\right|>\frac{\ln\rho^{-1}_n}{2}\right\}\\
&\leq\mathbb{P}_{*-\delta_0}\left\{
\sigma\sup\limits_{t\in[0,T_n]}M_n(t)>\exp\left\{\frac{(\ln\rho^{-1}_n)^2}{8\sigma^2T_n
M}\right\}\right\}\\
&\leq\exp\left\{-\frac{(\ln\rho^{-1}_n)^2}{8\sigma^2T_n M}\right\},
\end{aligned}
\end{equation}
where the last inequality is derived by the exponential martingale
inequality. \eqref{aux-1} and \eqref{exponential martingale} implies
the desired result.

\end{proof}

\begin{lemma}\label{lemma V} For any sufficiently small
$\delta_0>0$, let
$$
\overline{V}_{m,\rho}:=\inf\limits_{T\leq m}\inf\limits_{\phi\in
C([0,T]),\phi(T)=\rho}I_{*-\delta_0, T}(\phi).
$$
If $\overline{V}_{m,\rho}<\infty$, then there is a decreasing
$\phi(\cdot)\in C([0,T])$ for some $T\leq m$ such that
$\phi(T)=\rho$ for the first time and
\begin{equation*}\label{V m}
\overline{V}_m=I_{*-\delta_0,T}(\phi),
~\phi(t)\in[\rho,*-\delta]~for~any~t\in[0,T].
\end{equation*}
\end{lemma}
\begin{proof}
Since $\{\phi\in C([0,T]),\phi(T)=\rho\}$ is the closed set of
$C([0,T])$ and $I_{*-\delta_0, T}(\cdot)$ is a good rate function,
there exists a $\phi_T\in C([0,T]), T\leq m$ such that
$$
\inf\limits_{\phi\in C([0,T]),\phi(T)=\rho}I_{*-\delta_0,
T}(\phi)=I_{*-\delta_0, T}(\phi_T).
$$
Therefore, there is a sequence of $\{T_n, n\geq1\}$ and $\{\phi_n,
n\geq1\}$ such that $T_n\leq m$, $\phi_n\in C([0,T_n])$ and
$$
\phi_n(T_n)=\rho, ~I_{*-\delta_0,
T_n}(\phi_n)\to\overline{V}_{m,\rho}.
$$

Define
$$
\tau_n=\inf\{t\geq0;\phi_n(t)=\rho\}.
$$
Then $\tau_n\leq T_n$ and consider $\{\phi_n(t), t\in[0,\tau_n]\}$.
Since $\phi_n(\tau_n)=\rho$, $\overline{V}_m\leq
I_{*-\delta,\tau_n}(\phi_n)\leq I_{*-\delta_0, T_n}(\phi_n)$.
Therefore, we may assume that $\phi_n(T_n)=\rho$ for the first time
without loss of generality.

Similarly, define
$$
\widetilde{\tau}_n=\sup\{0\leq t\leq T_n;\phi_n(t)=*-\delta\}.
$$
Consider $\{\phi_n(t), t\in[\widetilde{\tau}_n,T_n]\}$. Since
$\phi_n(\widetilde{\tau}_n)=*-\delta$ and $\phi_n(T_n)=\rho$, we may
construct by homogeneity a trajectory
$\{\widetilde{\phi}_n(t),t\in[0,T_n-\widetilde{\tau}_n]\}$ such that
$\widetilde{\phi}_n(0)=*-\delta$,
$\widetilde{\phi}_n(T_n-\widetilde{\tau}_n)=\rho$ and
$\overline{V}_{m,\rho}\leq
I_{*-\delta_0,T_n-\widetilde{\tau}_n}(\widetilde{\phi}_n)\leq
I_{*-\delta_0, T_n}(\phi_n)$. Therefore, we may also assume that
$\phi_n(t)\leq*-\delta$ for any $t\in[0,T_n]$ without loss of
generality.

In all, there is a sequence of $\{T_n, n\geq1\}$ and $\{\phi_n,
n\geq1\}$ such that $T_n\leq m$, $\phi_n\in C([0,T_n])$,
$\phi_n(T_n)=\rho$, $I_{*-\delta_0, T_n}(\phi_n)\to\overline{V}_m$,
and $\phi_n(t)\in[\rho,*-\delta]$ for $t\in[0,T_n]$.

Note that $T_n\leq m$, we may assume that $T_n\uparrow T\leq m$
without loss of generality. For $t\in[T_n,T]$, let $u_n(t)=0$ and
$\phi_n(t)=\rho+\int^t_{T_n}\phi_n(s)(N-\mu-\gamma-\beta\phi_n(s))ds$.
Then $\phi_n(T)\to\rho$ as $n\to\infty$, $I_{*-\delta,
T}(\phi_n)=I_{*-\delta, T_n}(\phi_n)$ and
$\phi_n(t)\in[\rho,*-\delta]$ for $t\in[0,T]$ if $\rho$ is
sufficiently small.

Since $\overline{V}_{m,\rho}<\infty$,
$I_{*-\delta_0,T}(\phi_n)=\frac{\int^{T}_0 |u_n(t)|^2dt}{2}\leq
\overline{V}_{m,\rho}+1$, we may assume that $\{\phi_n, n\geq1\}$
converges to $\phi$ in $C([0,T])$. Alike the proof of Lemma 1.4.17
in \cite{Deuschel}, we could show that $\phi_n\to\phi_n$ in
$C([0,T])$ and $\phi_n$ converges weakly to $\phi_n$ in $H^1_T$.
Therefore, $u_n$ converges weakly to $u$ in $L^2([0,T])$, where
\begin{equation}\label{phi}
\phi_t=*-\delta_0+\int^t_0
\phi(s)\left(N-\mu\gamma-\beta\phi(s)\right)ds+\int^t_0\phi(s)\left(N-\phi(s)\right)u(s)ds.
\end{equation}
By Banach-Steinhaus Theorem,
$$
I_{*-\delta,T}(\phi)\leq\liminf_{n\to\infty}I_{*-\delta,T}(\phi_n)=\overline{V}_m.
$$
Obviously, $I_{*-\delta,T}(\phi)\geq\overline{V}_m$. Therefore,
$\overline{V}_m=I_{*-\delta,T}(\phi)$, where $T\leq m$,
$\phi(T)=\rho$ for the first time and $\phi(t)\in[\rho,*-\delta]$
for any $t\in[0,T]$.

In fact, we may assume that $\phi(\cdot)$ is nonincreasing in
$[0,T]$. Otherwise, there are $0\leq t_1<t_2\leq T$ such that
$\phi(t_1)<\phi(t_2)$. Since $\phi(\cdot)$ is continuous and
$\phi(T)=\rho$, there exists a $t_3>t_2$ such that
$\phi(t_3)=\phi(t_1)$. If $u(t)\equiv0$ a.s. for $t\in[t_1,t_3]$,
then $\phi\in[\rho,*-\delta]$ and
$\phi(t)=\phi(t_1)+\int^t_{t_1}\phi(s)(N-\mu-\gamma-\beta\phi(s))ds$
is increasing in $[t_1,t_3]$. This contradicts the assumption
$\phi(t_2)>\phi(t_3)$. This means that $u\neq0$ a.s. in $[t_1,t_3]$.
We could omit the time between $t_1$ and $t_3$, and splice the
trajectory in $[0,t_1]$ with the trajectory in $[t_3,T]$, and get a
new trajectory $\widetilde{\phi}(\cdot)$ in $C([0, T-(t_3-t_1)])$
such that $\widetilde{\phi}(0)=*-\delta$ and
$\widetilde{\phi}(T-(t_3-t_1))=\rho$ which satisfies
\begin{equation*}\label{cost function}
d\widetilde{\phi}(t)=\widetilde{\phi}(t)(\beta
N-\mu-\gamma-\beta\widetilde{\phi}(t))dt+\widetilde{\phi}(t)(N-\widetilde{\phi}(t))\widetilde{u}(t)dt,
\end{equation*}
in $[0, T-(t_3-t_1)]$, where $\widetilde{u}\in L^2\left([0,
T-(t_3-t_1)]\right)$ is defined according to $u$ by splice.  Since
$u\neq0$ a.s. in $[t_1,t_3]$, thus
$\frac{\int^{T-(t_3-t_1)}_0|\widetilde{u}(s)|^2ds}{2}<\frac{\int^T_0|u(s)|^2}{2}=\overline{V}_{m,\rho}$,
which contradicts the definition of $\overline{V}_{m,\rho}$.
Therefore, we may assume that $\phi(\cdot)$ is decreasing in
$[0,T]$.

\end{proof}

\begin{proposition}\label{proposition} For any $0<\rho<y$,
\begin{equation*}
\lim_{\rho\to0}\overline{V}_\rho=\lim_{\rho\to0}\overline{V}_{-\rho}=\infty.
\end{equation*}
\end{proposition}
\begin{proof}
We will give the proof of $V_\rho$, and the same method holds for
$\overline{V}_{-\rho}$.

Note that $\overline{V}_\rho$ is nondecreasing as $\rho\to0$.
Therefore, if
\begin{equation}\label{V 0}
\overline{V}:=\lim\limits_{\rho\to0}\overline{V}_\rho<\infty,
\end{equation}
then for sufficiently small $\rho>0$, we have
$\overline{V}_{\rho}\leq \overline{V}<\infty$.

Let $u(t)\equiv-\frac{2r_0\beta^2}{2r_0\beta+\mu+\gamma}$, where
$r_0$ is a fixed constant such that $0<2r_0<*$, and
\begin{equation*}
\begin{aligned}
\phi_t&=*+\int^t_0
\phi(s)\left(N-\mu\gamma-\beta\phi(s)\right)ds+\int^t_0\phi(s)\left(N-\phi(s)\right)u(s)ds
\\
&=*+\int^t_0\frac{\beta(\mu+\gamma)}{2r_0\beta+\mu+\gamma}\phi(s)\left(
*-2r_0-\phi(s)\right)dt.
\end{aligned}
\end{equation*}
Then there exist positive and sufficiently small $\delta_0$ and
$t_0$ such that $\phi(t_0)=*-\delta_0$ and
$\frac{\int^{t_0}_0|u(s)|^2ds}{2}<\frac{\overline{V}}{2}$, which
implies that
\begin{equation}\label{V *}
V(*,*-\delta_0)<\frac{\overline{V}}{2}.
\end{equation}

Note that $V(*,\rho)\geq V(*,*-\delta_0)+V(*-\delta_0,\rho)$, thus
\eqref{V 0} and \eqref{V *} implies that
\begin{equation}\label{V1}
\lim_{\rho\to0}V(*-\delta_0, \rho)\leq\frac{\overline{V}}{2}<\infty.
\end{equation}

Note that
\begin{equation*}\label{V2}
V(*-\delta_0, \rho)=\inf\limits_{m>0}\overline{V}_{m,\rho},
\end{equation*}
thus Lemma \ref{lemma V} and \eqref{V1} implies that there exists a
sequence of  $\rho_n\to0$ as $n\to\infty$, and there is a
$I_{*-\delta_0, T_n}(\phi_n)$ such that $\phi_n(\cdot)$ is
decreasing, contained in $[\rho_n,*-\delta_0]$ and
\begin{equation}\label{I n}
\sup\limits_{n}I_{*-\delta_0,
T_n}(\phi_n)=\sup\limits_{n}\frac{\int^{T_n}_0|u_n(t)|^2dt}{2}<\infty,
\end{equation}
where the relationship between $\phi_n$ and $u_n$ is defined as
\eqref{phi}.

By Lemma \ref{lemma V}, $\phi'_n(t)\leq0$ and
$\phi_n(t)\in[\rho,*-\delta_0]$ for $t\in[0,T_n]$. Then
$$
\phi_n(t)\left(\tilde{\beta}N-\mu-\gamma-\tilde{\beta}\phi(t)\right)\leq0,
$$
where  $\tilde{\beta}=\beta+u_n(t)$.

By computation, for any $t\in[0,T_n]$,
\begin{equation*}
\begin{aligned}
u_n(t)&\leq\frac{\mu+\gamma}{1-\phi(t)}-\beta
\\
&=\frac{\beta(\phi(t)-*)}{N-\phi(t)}\leq-\frac{\beta\delta_0}{N-*}.
\end{aligned}
\end{equation*}
Therefore, $|u_n(t)|\geq \frac{\beta\delta_0}{N-*}$ for
$t\in[0,T_n]$. So
$I_{*-\delta_0,T_n}(\phi_n)\geq\frac{\beta^2\delta^2_0
T_n}{2(N-*)^2}$.

By \eqref{I n}, $\sup\limits_{n\to\infty}T_n<\infty$. Similarly, we
may also assume that for some $M>0$
\begin{equation}\label{I M}
\sup\limits_{n}I_{*-\delta_0,T_n}(\phi_n)\leq M.
\end{equation}

By the discussion above, $\phi_n\in[\rho_n,*-\delta_0]$ for
$t\in[0,T_n]$, $\sup\limits_{n\to\infty}T_n<\infty$ and
$\sup\limits_{n}I_{*-\delta_0,T_n}(\phi_n)\leq M$. On the other
hand, by the lower bound of the large deviation principle,
\begin{align*}\label{results from the lower bound}
\liminf_{\sigma\to0}\sigma^2\log\mathbb{P}_{*-\delta_0}\left\{\inf_{t\in[0,T_n]}I^\sigma(t)<2\rho_n,
I^\sigma(t)\in(\frac{\rho_n}{2},*-\frac{\delta_0}{2})\right\}\\
\geq-\inf\left\{I_{*-\delta_0,T_n}(\phi);\inf_{t\in[0,T_n]}\phi<2\rho_n,
\phi\in(\frac{\rho_n}{2},*-\frac{\delta_0}{2})\right\}.
\end{align*}
Then by Lemma \ref{LDP lower bound}, for any
$\inf\limits_{t\in[0,T_n]}\phi<2\rho_n,
\phi\in(\frac{\rho_n}{2},*-\frac{\delta_0}{2})$,
\begin{equation*}
I_{*-\delta_0,T_n}(\phi)\geq\frac{(\ln\rho^{-1}_n)^2}{8\sigma^2T_n
M}.
\end{equation*}
Especially, $\{\phi_n\}$ satisfies the above conditions and then
$I_{*-\delta_0,T_n}(\phi_n)\geq\frac{(\ln\rho^{-1}_n)^2}{8\sigma^2T_n
M}$.

Since $\sup\limits_n T_n<\infty$ and
$\lim\limits_{n\to\infty}\rho_n=0$, we have
$$\lim_{n\to\infty}I_{*-\delta_0,T_n}(\phi_n)=\infty,$$
which contradicts \eqref{I M}. Therefore, The proof is completed.
\end{proof}
\begin{remark}
By the definition of $\overline{V}_\rho$, we may prove that
$\overline{V}_0=\lim\limits_{\rho\to0}\overline{V}_\rho$. Therefore,
what we have to do is just to prove that $\overline{V}_0=\infty$.
Note that $\overline{V}_0<\infty$ is equivalent to the existence of
$0<T<\infty$ and $u(\cdot)\in L^2([0,T])$ such that $\phi(T)=0$ and
for $t\in[0,T]$,
\begin{equation*}\label{support thm}
\phi_t=*+\int^t_0
\phi(s)\left(N-\mu\gamma-\beta\phi(s)\right)ds+\int^t_0\phi(s)\left(N-\phi(s)\right)u(s)ds.
\end{equation*}
Therefore, one may be initialized to investigate the positivity for
the density of $I^\sigma(T)$ at $0$ with the initial condition
$I^\sigma(0)=*$ (see \cite{AK}, \cite{BL}, \cite{SV} and references
therein). But it should be careful that the diffusion coefficient is
degenerate at $0$ and a simple computation implies that the
H\"{o}rmander condition (we refer \cite{Norris} and \cite{SV} for
reference) are not satisfied. Therefore, the support theorems can
not be applied directly. Here, we adapt the analysis of variation
and the large deviation principle to get the desired results.

\end{remark}

\section{Proof of Main results}
\label{sec:eqn} \setcounter{equation}{0}

{\bf  Proof of Theorem \ref{main results}} ~(1) Firstly, note that
\begin{align*}
&\mathbb{P}\left\{\tau^\sigma_y>e^{\frac{\overline{V}_y\wedge
\overline{V}_\rho+\delta}{\sigma^2}} \right\}\\
&\leq \mathbb{P}\left\{\tau^\sigma_y\wedge\tau^\sigma_\rho>
e^{\frac{\overline{V}_y\wedge \overline{V}_\rho+\delta}{\sigma^2}}
\right\}+\mathbb{P}\left\{\tau^\sigma_\rho\wedge\tau^\sigma_{-\rho}\leq
e^{\frac{\overline{V}_y\wedge \overline{V}_\rho+\delta}{\sigma^2}}
\right\}\\
&\leq \mathbb{P}\left\{\tau^\sigma_y\wedge\tau^\sigma_\rho>
e^{\frac{\overline{V}_y\wedge \overline{V}_\rho+\delta}{\sigma^2}}
\right\}+\mathbb{P}\left\{e^{\frac{\overline{V}_\rho\wedge
\overline{V}_{-\rho}-\delta}{\sigma^2}}\leq\tau^\sigma_\rho\wedge\tau^\sigma_{-\rho}\leq
e^{\frac{\overline{V}_y\wedge \overline{V}_\rho+\delta}{\sigma^2}}
\right\}\\
&+\mathbb{P}\left\{\tau^\sigma_\rho\wedge\tau^\sigma_{-\rho}<
e^{\frac{\overline{V}_\rho\wedge \overline{V}_{-\rho}-\delta}{\sigma^2}} \right\}\\
&:=P^1_{\sigma,\rho}+P^2_{\sigma,\rho}+P^3_{\sigma,\rho}.
\end{align*}
In the following paragraph, we will give their estimation
respectively.

In the model of \eqref{SDE I}, the equilibrium of the deterministic
system is $*$ and consider the boundary $\partial G=\{\rho, y\}$ for
$y>*$ and $\rho$ is sufficiently small. It can be verified that the
Assumptions (A-1)$-$(A-4) are satisfied. We will give the detail of
them below.

The Assumption (A-1) and (A-2) are easily verified. For (A-3), let
$u_t=u$ sufficiently large such that
$N-\frac{\mu+\gamma}{\beta+u}>y$, then for the deterministic system
\begin{equation}\label{cost function}
\phi(t)=*+\int^t_0 \phi(s)(\beta
N-\mu-\gamma-\beta\phi(s))ds+\int^t_0\phi(s)(N-\phi(s))u(s)ds,
\end{equation}
there exists a $T>0$ such that $\phi(T)=y$ and
$\overline{V}_y\leq\frac{\int^T_0 u^2(s)ds}{2}<\infty$. Meanwhile,
for any $x_1, x_2$ sufficiently close to each other in the
neighborhood of $y$, there exists $T(\rho)$ such that
$\phi(T(\rho))=x_2$, \eqref{cost function} holds and $T(\rho)\to0$
as $\rho\to0$. When $x_1, x_2$ are sufficiently close to each other
in the neighborhood of $\rho$ or $*$, we can get the same results.
Then Assumption (A-4) holds.

Therefore, for any $y\in(*,N)$, Theorem \ref{Thm 5.7.11} implies
\begin{equation*}\label{auxiliary estimation}
\lim_{\sigma\to0}P^1_{\sigma,\rho}=\lim_{\sigma\to0}\mathbb{P}\left\{\tau^\sigma_y\wedge\tau^\sigma_\rho>e^{\frac{\overline{V}_y\wedge
\overline{V}_\rho+\delta}{\sigma^2}}\right\}=0.
\end{equation*}
Similarly,
\begin{equation*}
\lim_{\sigma\to0}P^3_{\sigma,\rho}=0.
\end{equation*}
What is left is the estimation of $P^2_{\sigma,\rho}$. In fact, by
Proposition \ref{proposition},
\begin{equation*}
\lim_{\rho\to0}\overline{V}_\rho=\lim_{\rho\to0}\overline{V}_{-\rho}=\infty.
\end{equation*}
Thus let $\rho$ be sufficiently small such that
$\overline{V}_\rho\wedge
\overline{V}_{-\rho}>\overline{V}_y+2\delta$, which implies
\begin{equation*}
\lim_{\sigma\to0}P^2_{\sigma,\rho}=0.
\end{equation*}
Therefore,
$\lim\limits_{\sigma\to0}\mathbb{P}\left\{\tau^\sigma_y>e^{\frac{\overline{V}_y\wedge
\overline{V}_\rho+\delta}{\sigma^2}}\right\}=0$.

Since $\overline{V}_y\wedge \overline{V}_\rho=\overline{V}_y$ for
sufficiently small $\rho$,
$$
\lim_{\sigma\to0}\mathbb{P}\left\{\tau^\sigma_y>e^{\frac{\overline{V}_y+\delta}{\sigma^2}}
\right\}
=\lim_{\sigma\to0}\mathbb{P}\left\{\tau^\sigma_y>e^{\frac{\overline{V}_y\wedge
\overline{V}_\rho+\delta}{\sigma^2}}\right\}=0.
$$
The proof of upper bound ends.

Now, we turn to the proof of the lower bound.
\begin{align*}
&\mathbb{P}\left\{\tau^\sigma_y\leq e^{\frac{\overline{V}_y\wedge
\overline{V}_\rho-\delta}{\sigma^2}} \right\}\\
&\leq\mathbb{P}\left\{\tau^\sigma_y\wedge\tau^\sigma_\rho\leq
e^{\frac{\overline{V}_y\wedge \overline{V}_\rho-\delta}{\sigma^2}}
\right\}+\mathbb{P}\left\{\tau^\sigma_\rho\wedge\tau^\sigma_{-\rho}\leq
e^{\frac{\overline{V}_y\wedge \overline{V}_\rho-\delta}{\sigma^2}}
\right\}\\
&\leq \mathbb{P}\left\{\tau^\sigma_y\wedge\tau^\sigma_\rho\leq
e^{\frac{\overline{V}_y\wedge \overline{V}_\rho-\delta}{\sigma^2}}
\right\}+\mathbb{P}\left\{e^{\frac{\overline{V}_\rho\wedge
\overline{V}_{-\rho}-\delta}{\sigma^2}}\leq\tau^\sigma_\rho\wedge\tau^\sigma_{-\rho}\leq
e^{\frac{\overline{V}_y\wedge \overline{V}_\rho\wedge
\overline{V}_{-\rho}-\delta}{\sigma^2}} \right\}
\\
&+\mathbb{P}\left\{\tau^\sigma_\rho\wedge\tau^\sigma_{-\rho}<
e^{\frac{\overline{V}_\rho\wedge \overline{V}_{-\rho}-\delta}{\sigma^2}} \right\}\\
&:=Q^1_{\sigma,\rho}+Q^2_{\sigma,\rho}+Q^3_{\sigma,\rho}.
\end{align*}
The lower bound can be proved in the same way.

Now, we turn to the proof of $0<\overline{V}(y)<\infty$. Since
$0<y<N$, let $u(t)\equiv u$ be sufficiently large such that
$N-\frac{\mu+\gamma}{\beta+u}>y$, then by (1.2) in \cite{Gray},
there exists $T>0$ such that $\phi(T)=y$, where $\phi(t)=x+\int^t_0
\phi(s)(\beta
N-\mu-\gamma-\beta\phi(s))ds+\int^t_0\phi(s)(N-\phi(s))u(s)ds$.
Thus, by the definition of $\overline{V}_y$,
$\overline{V}_y<\frac{u^2T}{2}<\infty$.

Let $*<\delta<y$, note that $\overline{V}_y\geq V(y-\delta,y)$. Then
$\overline{V}_y=0$ implies $V(y-\delta,y)=0$. Then there are two
sequences of $\{T_n,n\geq1\}$ and $\{\phi_n,\geq1\}$ such that
$I_{y-\delta,T_n}(\phi_n)\to0$, where $\phi_n(0)=y-\delta$,
$\phi_n(T_n)=y$,
$\frac{\int^{T_n}_0|u_n(t)|^2}{2}=I_{y-\delta,T_n}(\phi_n)\to0$, and
\begin{equation*}
\phi_n(t)=y-\delta+\int^t_0\phi_n(s)\left(\beta
N-\mu-\gamma-\beta\phi_n(s)ds\right)ds+\int^t_0\phi_n(s)\left(N-\phi_n(s)\right)u_n(s)ds
\end{equation*}
for all $t\in[0,T_n]$. Alike the proof of Proposition
\ref{proposition}, we could show that $T_n\to 0$. It is easy to
prove that $\phi_n(T_n)$ converges to $y-\delta$. This contradicts
the fact that $\phi_n(T_n)=y$. Therefore, $\overline{V}_y>0$ for
$*<y<N$.

(2) For any $0<y<*$, let $T_y=\inf\left\{t\geq0;I^0(t)=y\right\}$,
then $T_y<\infty$, and for any $\delta>0$, we also define
$d(\delta):=\min\left\{I^0(T_y+\delta)-y,
y-I^0(T_y-\delta)\right\}>0$ accordingly. Since the coefficients of
$I^\sigma(\cdot)$ and $I^0(\cdot)$ are of uniformly bounded Liptsiz,
there exists $M>0$ such that
\begin{equation*}
\begin{aligned}
G^\sigma_t\leq M\int^t_0G^\sigma_sds+\sigma\sup_{s\leq
t}\int^s_0\left|I^\sigma(s)\left(N-I^\sigma(s)\right)dB_s\right|,
\end{aligned}
\end{equation*}
where $G^\sigma_t=\sup\limits_{s\leq t}|I^\delta(s)-I^0(s)|$. Then
by Gronwall's inequality, there exists $M'>0$ such that
\begin{equation*}
G^\sigma_{T_y+\delta}\leq M'\sigma\sup_{s\leq
T_y+\delta}\int^s_0\left|I^\sigma(s)\left(N-I^\sigma(s)\right)dB_s\right|,~a.s.
\end{equation*}
Therefore, there are $M''>0$ and $M'''>0$ such that
\begin{equation*}
\begin{aligned}
\mathbb{P}\left\{\sup_{t\in[0,T_y+\delta]}|I^\sigma(t)-I^0(t)|>\frac{d(\delta)}{2}\right\}
&\leq\mathbb{P}\left\{\sigma\sup_{s\leq
T_y+\delta}\int^s_0\left|I^\sigma(s)\left(N-I^\sigma(s)\right)dB_s\right|>M''d(\delta)\right\}\\
&\leq
M'''\sigma^2E\int^{T_y+\delta}_0\left|I^\sigma(s)\left(N-I^\sigma(s)\right)\right|^2ds,
\end{aligned}
\end{equation*}
where the last inequality is derived by the B-D-G inequality for
continuous martingales.

Therefore,
\begin{equation*}
\begin{aligned}
&\lim_{\sigma\to0}\mathbb{P}\left\{\sup_{t\in[0,T_y+\delta]}|I^\sigma(t)-I^0(t)|>\frac{d(\delta)}{2}\right\}=0.
\end{aligned}
\end{equation*}
By the definition of $d(\delta)$ and $T_y$,
$\sup\limits_{t\in[0,T_y+\delta]}|I^\sigma(t)-I^0(t)|\leq\frac{d(\delta)}{2}$
implies that $T_y-\delta<\tau^\sigma_y<T_y+\delta$.

Hence,
\begin{equation*}
\begin{aligned}
\lim_{\sigma\to0}\mathbb{P}\left\{T_y-\delta<\tau^\sigma_y<T_y+\delta\right\}=1.
\end{aligned}
\end{equation*}

(3) Firstly, note that $T_y$ is increasing, then
$\lim\limits_{y\to*}T_y=\sup\limits_{y<*}T_y=:T_0$. If $T_0<\infty$,
by the definition of $T_y$ and the continuity of $I^0(\cdot)$,
$I^0(T_0)=*$, which contradicts the trajectory property of
$I^0(\cdot)$. Therefore, $\lim\limits_{y\to*}T_y=\infty.$

Since for any $y<*$, $\tau^\sigma_y<\tau^\sigma_*$ and
\begin{equation*}
\lim_{\sigma\to0}\mathbb{P}_x\left\{T_y-\delta<\tau^\sigma_y<T_y+\delta\right\}=1,
\end{equation*}
for any $M>0$, we have
\begin{equation}\label{tau M}
\lim_{\sigma\to0}\mathbb{P}_x\left\{\tau^\sigma_y>M\right\}=1,
\end{equation}
i.e., $\lim\limits_{\sigma\to0}\tau^\sigma_y=\infty$ in probability.

Next, We will show that $\lim\limits_{y\downarrow
*}\overline{V}_y=0$. Let $u_t=u$ sufficiently large such that
$N-\frac{\mu+\gamma}{\beta+u}>*$, then there exists a trajectory
$$
\phi(t)=*+\int^t_0\phi(s)\left[(\beta+u)
N-\mu-\gamma-(\beta+u)\phi(s)\right]dt
$$
and $\phi(T_y)=y$ such that $T_y\to0$ as $y\to*$. Therefore,
$\lim\limits_{y\downarrow
*}\overline{V}_y=0$.

For any $\delta>0$, let $y$ be sufficiently close to $y^*$ and
$y>y^*$ such that $\overline{V}_y\leq\frac{\delta}{2}$, then
\begin{equation}\label{tau 0 1}
\limsup_{\sigma\to0}\mathbb{P}\left\{\sigma^2\log
\tau^\sigma_*>\delta\right\}\leq\limsup_{\sigma\to0}\mathbb{P}\left\{\sigma^2\log
\tau^\sigma_y-\overline{V}_y>\frac{\delta}{2}\right\}=0,
\end{equation}
where the last inequality is derived by (1) in Theorem \ref{main
results}. Therefore, $\limsup\limits_{\sigma\to0}\sigma^2\log
\tau^\sigma_*\leq0$.

Let $x<y<*$, then $\tau^\sigma_*>\tau^\sigma_y$ and
$\lim\limits_{\sigma\to0}\tau^\sigma_y=T_y$, where $0<T_y<\infty$.
Therefore,
\begin{equation}\label{tau 0 2}
\begin{aligned}
&\limsup_{\sigma\to0}\mathbb{P}\left\{\sigma^2\log
\tau^\sigma_*<-\delta\right\}\\
&\leq\limsup_{\sigma\to0}\mathbb{P}\left\{\sigma^2\log
\tau^\sigma_y<-\delta,
\tau^\sigma_y\geq\frac{T_y}{2}\right\}+\limsup_{\sigma\to0}\mathbb{P}\left\{\tau^\sigma_y<\frac{T_y}{2}\right\}=0.
\end{aligned}
\end{equation}
where the last inequality is derived by \eqref{tau M} (2) in Theorem
\ref{main results}.

Therefore, for any $\delta>0$, \eqref{tau 0 1} and \eqref{tau 0 2}
implies $\limsup\limits_{\sigma\to0}\mathbb{P}\left\{|\sigma^2\log
\tau^\sigma_*|>\delta\right\}=0$.
\\

The rest proof of (4)-(6) is similar to (1)-(3), so we omit it.
Thus, the proof is completed.


\begin{thebibliography}{99}


\bibitem{AK}S. Aida, S. Kusuoka, D. Strook, On the support of Wiener
functionals. In: Elworthy, K.D., Ikeda, N.  (Eds.), Asymptotic
Problems in Probability Theory: Wiener Functionals and Asymptotic.
Pitman Research Notes in Mathematical Series, 284, Longman Scient.
Tech., 3-34, 1993.




\bibitem{BL}G. Ben Arous, R. LLeandre, D\'{e}croissance exponentielle du
noyau de la chaleur sur la diagonale (II).  Probab. Theory Related
Fields. 90 (1991) 377-402.





\bibitem{Day89}M.V. Day, Boundary local time and small parameter
exit problems with characteristic boundaries, SIAM J. Math. Anal.,
20 (1989) 222-248.

\bibitem{Day92}M.V. Day, Conditional exits for small noise diffusions with characteristic boundary, Ann. Probab.,
20 (1992) 1385-1419.


\bibitem{Dembo} A. Dembo, O. Zeitoini,
Large deviations techniques and applications, Second Edition,
Sringer, 1998.

\bibitem{Deuschel}J. Deuschel, D.W. Stroock, Large deviations,
Academic Press. Inc, 1989.


\bibitem{Duk87}P. Dupuis, H.J. Kushner, Stochastic systems with
small noise, analysis and simulation; a phase locked loop example,
SIAM J. APPl. Math., 47 (1987) 643-661.

\bibitem{Fre85b} M.I. Freidlin, Limit theorems for large deviations
and reaction-diffusion equations, Ann. Probab., 13 (1985) 639-675.

\bibitem{Gray}
A. Gray, D. Greenhalgh, L. Hu, X. Mao, J. Pan, A stochastic
differential equation SIS epidemic model, SIAM. J Appl Math. 71
(2011) 876-902.



\bibitem{Liu}
H. Liu, Q. Yang, D. Jiang, The asymptotic behavior of stochastically
perturbed DI SIR epidemic models with saturated incidences,
Automatica, 48:5 (2012) 820-825.


\bibitem{Norris}
J. Norris, Simplified Malliavin calculus. In: S\'{e}minaire de
probabilitiLes XX, Lecture Notes in Mathematics, 1204. Springer,
101-130, 1986.








\bibitem{SV}D.W. Stroock, S.R.S. Varadhan, On the support of di8usion processes
with applications to the  strong maximum principle. Proceedings of
the Sixth Berkeley Symposium on Mathematical Statistics  and
Probability, Vol. III. Univ. Cal. Press, Berkeley, 333-360, 1972.



\bibitem{Yang}
Q. Yang, X. Mao, Extinction and recurrence of multi-group SEIR
epidemic models with stochastic perturbations, Nonlinear Anal-Real.
14:3 (2013) 1434-1456.


\end{thebibliography}
\end{document}